\documentclass{aims}
\usepackage{amsmath}
  \usepackage{paralist}
  \usepackage{graphics} 
  \usepackage{epsfig} 
\usepackage{graphicx}  \usepackage{epstopdf}
 \usepackage[colorlinks=true]{hyperref}
\usepackage{tikz}
\usetikzlibrary{decorations.markings}
\hypersetup{urlcolor=blue, citecolor=red}

\def\currentvolume{X}
 \def\currentissue{X}
  \def\currentyear{200X}
   \def\currentmonth{XX}
    \def\ppages{X--XX}
     \def\DOI{10.3934/xx.xx.xx.xx}

\newtheorem{theorem}{Theorem}[section]

\newtheorem{lemma}[theorem]{Lemma}

\theoremstyle{definition}

\newtheorem{remark}{Remark}

\newtheorem{hypothesis}{Hypothesis}[section]

\title[Acetabularia Model] 
      {Dynamic Transitions and Stability for the Acetabularia Whorl Formation}

\author[Yiqiu Mao, Dongming Yan and ChunHsien Lu]{}


 \email{ymao@indiana.edu}
 \email{dmyan@zufe.edu.cn}
 \email{lu58@umail.iu.edu}

\thanks{$^*$ Corresponding author: Yiqiu Mao}

\begin{document}
\maketitle

\centerline{\scshape Yiqiu Mao$^*$}
\medskip
{\footnotesize
 \centerline{Department of Mathematics}
   \centerline{Indiana University}
   \centerline{ Bloomington,IN 47405, USA}
} 

\medskip

\centerline{\scshape Dongming Yan}
\medskip
{\footnotesize
 \centerline{ School of Data Sciences}
   \centerline{Zhejiang University of Finance $\&$ Economics}
   \centerline{Hangzhou, 310018, PR China}
}

\medskip

\centerline{\scshape ChunHsien Lu}
\medskip
{\footnotesize
	\centerline{Department of Mathematics}
	\centerline{Indiana University}
	\centerline{ Bloomington,IN 47405, USA}
}
\bigskip

 \centerline{(Communicated by the associate editor name)}

\begin{abstract}
Dynamical transitions of the Acetabularia whorl formation caused by outside calcium concentration is carefully analyzed using a chemical reaction diffusion model on a thin annulus. Restricting ourselves with Turing instabilities, we found all three types of transition, continuous, catastrophic and random can occur under different parameter regimes. Detailed linear analysis and numerical investigations are also provided. The main tool used in the transition analysis is Ma \& Wang's dynamical transition theory including the center manifold reduction. 
\end{abstract}

\section{Introduction}	
Acetabularia acetabulum is an important species for studying the morphogenesis during development for a single cell organism. In this article we will focus on the vegetative whorl generation along the stalk. Biological details of the Acetabularia development can be found in various review articles, for example \cite{dumais2000whorl}\cite{dumais2000acetabularia}.

Several different models have been put forward to explain the whorl generation phenomenon. We note here since Acetabularia is a single cell organism and evidence has been shown that without nucleus the stem is still capable of growing a cap \cite{goodwin1994leopard}, the whorl generation is largely a local dynamics involving chemical reactions and cytoskeleton hence hope is given for a simple model with differential equations not involving genetics or intracellular signaling. Chemical as well as mechanical models \cite{martynov1975morphogenetic}\cite{goodwin1985tip}\cite{harrison1992reaction} have been proposed. 

In this article we will apply the model used by Murray \cite{murray2001mathematical} , which is a purely chemical model and a reaction diffusion equation involving two reactants, and the pattern formation is driven by the Turing instability of the system. Several observations have been shown to support the kinematic nature of our model: The spacing of the hair is affected by temperature \cite{harrison1981hair}, the spatial pattern of whorl hair is related to outside calcium concentration \cite{goodwin1984calcium}, and the whorl pattern is initiated simultaneously.

We will focus ourselves here with the effect of calcium concentration on the initiation of the whorl pattern. Experiments have shown \cite{goodwin1994leopard} that there is a threshold of calcium concentration of about $2$mM for the formation of whorls. Interestingly for external calcium above about 60mM, whorls will stop forming once again. Both facts can be shown in our model and details of the transition from stable stem to whorl formation when Calcium concentration crosses a critical value are explicitly calculated, using dynamical transition theory developed by Ma and Wang\cite{ma2005bifurcation}\cite{ma2014phase}.

More specifically, we obtain a specific interval of calcium concentration for the whorl pattern to occur. Then we proceed to classify the onset of pattern formation into three types, corresponding to the three types in Ma \& Wang's theory. We have shown that all three kinds of transition type can occur though in the case of thin stem wall only continuous type and catastrophic type can happen. We also show in the numerical section that the number of whorl hairs generated during transition can be quite irregular with underlying parameters. 

Besides the purpose of investigating the Acetabularia whorl formation, this paper can be viewed as an example for exploring dynamic behavior for reaction diffusion equations in an annular region. 

The article is organized as following: the modeling background is introduced in Section 2, along with the deviation equation and functional setup and a discussion of global existence; in Section 3 we analyze the linear problem, from properties of Bessel functions we derive the crossing eigenvectors must be two dimensional when stem walls are thin, exact classifications of parameter space are provided, along with the principle of exchange of stabilities; in Section 4 we focus on the dynamical transitions of the system and prove three types of transitions can occur in different cases; in Section 5 we use numerical tools to investigate the transition types and properties of critical eigenvectors; in Section 6 we derive some physical conclusions from our research.

\section{ Model equations } \label{2}
In this paper, we apply a model proposed by Murray \cite{murray2001mathematical}, which is an adaptation of a simple two species mechanism, the Schnackenberg (1979) system. The equations state as follows: 
\begin{align}
\left\{\begin{array}{ll}
\frac{\partial A}{\partial t}
=D_A\Delta A+k_1-k_2A+k_3A^2 B \\
\frac{\partial B}{\partial t}=D_B\Delta B+k_4-k_3A^2B.
\end{array}\right.
\end{align}
Where all $k$'s and $D_A,D_B$ are all positive parameters, $A$ and $B$ are functions of $r,\theta$ (or $\mathbf{x}$) and $t$ with the annulus domain defined by 
$$R_i\leq r\leq R_0, \quad 0\leq \theta < 2\pi.$$  $A,B$ define the density of two substances inside the annular growth region at the tip of mature Acetabularia. Here we assumed a reaction $2A+B\xrightarrow{k_3}3A$ took place. In our context, $B$ is calcium, and $A$ could be a molecule that is increased in the calcium presence but are constantly transformed to other substances at a rate of $k_2$. And they are generated by constant rate $k_1, k_4$ respectively for $A$ and $B$. We use the non-dimensionalization as follows:
\begin{align*}
&u=A\left(\frac{k_3}{k_2}\right)^{1/2},v=B\left(\frac{k_3}{k_2}\right)^{1/2}, t^*=\frac{D_A t}{R_i^2},\quad \mathbf{x}^*=\frac{\mathbf{x}}{R_i},\\
&d=\frac{D_B}{D_A},\quad a=\frac{k_1}{k_2}\left(\frac{k_3}{k_2}\right)^{1/2}, \lambda=\frac{k_4}{k_2}\left(\frac{k_3}{k_2}\right)^{1/2},R^2=\frac{R_i^2k_2}{D_A}.
\end{align*}

Then we have the reaction diffusion system after omitting the stars:
\begin{align}\label{cs}
\left\{\begin{array}{ll}
\frac{\partial u}{\partial t}
=\Delta u+R^2(a-u+u^2v) &\text{in}~~\Omega,\\
\frac{\partial v}{\partial t}=d\Delta v+R^2(\lambda-u^2v) &\text{in}~~\Omega,\\
u_r=v_r=0 &\text{on}~~\partial\Omega,
\end{array}\right.
\end{align}
where
the Laplacian
$$\Delta=\frac{\partial^2}{\partial r^2}
+\frac{1}{r}\frac{\partial}{\partial r}
+\frac{1}{r^2}\frac{\partial^2}{\partial \theta^2},$$
and ~$\Omega$ is the annular domain (we denote $\delta=\frac{R_0}{R_i}>1$)
$$1\leq r\leq \delta,~0\leq \theta<2\pi.$$ The $\delta$ is assumed to be very close to 1 as the stalks of Acetabularia have very thin walls. 

The inward flux constant $\lambda$ is directly proportional to outside calcium concentration because the normal intracellular concentration of calcium is extremely low compared to outside. Hence we will use $\lambda$ as a principle parameter and investigate the transition when the increase of $\lambda$ causes instability of the steady state.

It is easy to see that
\begin{align}\label{steady}
u_0=a+\lambda,~v_0=\frac{\lambda}{(a+\lambda)^2}
\end{align}
is a steady state of \eqref{cs}.
Consider deviation of \eqref{cs} from \eqref{steady}. To do so, we take
\begin{align}\label{deviation}
(u,~v)=(u'
+u_0,~
v'+v_0).
\end{align}
Dropping the primes, the system \eqref{cs} is then transformed into
\begin{align}\label{csdeviation}
\left\{\begin{array}{ll}
\frac{\partial u}{\partial t}
=~\Delta u+R^2\Big[\frac{\lambda-a}{a+\lambda}u
+(a+\lambda)^2v
+\frac{\lambda}{(a+\lambda)^2}u^2
+2(a+\lambda)uv+u^2v\Big],\\
\frac{\partial v}{\partial t}=d\Delta v-R^2\Big[\frac{2\lambda}{a+\lambda}u
+(a+\lambda)^2v+\frac{\lambda}{(a+\lambda)^2}u^2
+2(a+\lambda)uv+u^2v\Big],\\
u_r=v_r=0,~~r=1,\delta.
\end{array}\right.
\end{align}

Next, we convert the equation  \eqref{csdeviation}  into an abstract functional setting which is standard in the framework of dynamic transitions.
First, we let $X$ and $X_{1}$ be function spaces defined as follows:
\begin{align*}
&X~=L^2(\Omega)\times L^2(\Omega) ,\\&X_1=\Big\{(u,v)\in H^2(\Omega)\times H^2(\Omega)~\big|
~u_r=v_r=0~\text{on}~
r=1,\delta \Big\}.
\end{align*}
Then,
we define the operators $L_{\lambda}=-A+B_{\lambda}: X_1\rightarrow X$ and $G: X^{1/2}\to X$ ($X^{1/2}$ to be explained later) by
\begin{align}\label{21}
-A(U)=\left(\begin{array}{l}~\Delta u\\
d\Delta v
\end{array}\right),
~~~~
B_{\lambda}(U)=\left(\begin{array}{l}
~~R^2\frac{\lambda-a}{a+\lambda}u
+R^2(a+\lambda)^2v\\-R^2\frac{2\lambda}{a+\lambda}u-R^2(a+\lambda)^2v
\end{array}\right),
\end{align}
\begin{align}\label{22}\begin{array}{ll}
G(U)&=\left(\begin{array}{l}
\ \  R^2\big[\frac{\lambda}{(a+\lambda)^2}u^2
+2(a+\lambda)uv+u^2v\big]\\
-R^2\big[
\frac{\lambda}{(a+\lambda)^2}u^2
+2(a+\lambda)uv+u^2v\big]
\end{array}\right)\\
&=:\left(\begin{array}{l}
\ \ \ \gamma_{2}u^2+\gamma_{12}uv
+\gamma_3u^2v\\
-(\gamma_{2}u^2+\gamma_{12}uv
+\gamma_3u^2v
)\end{array}\right),
\end{array}
\end{align}
where $U=(u,v)$.

We thus obtain the equivalent operator equation of the evolution equation \eqref{csdeviation} as follows
\begin{equation}
\begin{aligned}\label{231}
&\frac{\text{d}U}{\text{d}t}=L_{\lambda} U+G(U),\\&U(0)=U_0.
\end{aligned}
\end{equation}
It is easy to see that $L_{\lambda}$ is a completely continuous field, meaning it's the sum of a linear homeomorphism $A: X_1\rightarrow X$ (with graph norm on $X_1$) and a compact operator $B_\lambda: X_1\rightarrow X$, which insures relatively nice spectral property for $L_\lambda$, see \cite{ma2005bifurcation} Chap 3. 

Since $H^1(\Omega)^2=X^{1/2}$ is continuously imbedded in $L^6$, thus it's also easy to check from \eqref{22} that the cubic $G$ is locally Lipschitz continuous from $X^{1/2}$ to $X$ which guarantees the local existence. Moreover $G(U)=o(||U||_{X^{1/2}})$ and is the sum of several continuous multilinear operators from $X_1$ to $X$. Hence the requirement on P527 of \cite{ma2014phase} is satisfied.

In \cite{you2007global}, Y. You proved a global attractor (hence global existence) exists for this equation for any initial value $U_0\in X$ with Dirichlet boundary conditions.

\section{Eigenvalues of Linearized Problem and Principle of
	exchange of stability (PES) }

The linear eigenvalue equations of \eqref{csdeviation} are given by
\begin{align}\label{ep}
\left\{\begin{array}{ll}
~\Delta u+R^2\Big[\frac{\lambda-a}{a+\lambda}u
+(a+\lambda)^2v\Big]
=\beta(\lambda) u,\\
d\Delta v-R^2\Big[\frac{2\lambda}{a+\lambda}u
+(a+\lambda)^2v\Big]=\beta(\lambda) v,\\
u_r=v_r=0,~~r=1,\delta.
\end{array}\right.
\end{align}

We first consider the eigenvalue problem \eqref{ep}.

Suppose the eigenvalue equation 
\begin{equation*}
\Delta u = \gamma u
\end{equation*}
with Neumann boundary condition in a annulus has following solutions: 
\begin{equation}
\label{laplacian eigenvalue problem on a disk}
\Delta B_k=-\lambda_k B_k
\end{equation} where $0=\lambda_0<\lambda_1\leq\lambda_2\leq...$.\\By using separation of variables $B=R(r)\Theta(\theta)$ in \eqref{laplacian eigenvalue problem on a disk}, it can be rewritten as: 
\begin{align}
&\Theta=\begin{cases}
1, \quad n=0,\\
\sin{n\theta} \text{ or }\cos{n\theta}, \quad n\geq 1.
\end{cases}\\
&R''+r^{-1}R'+(\lambda_k-\frac{n^2}{r^2})R=0,\quad R'(1)=R'(\delta)=0.\label{radial eigenvalue equation}
\end{align}
Hence the solution for $R$'s are Bessel functions $$R=J_n(\sqrt{\lambda_k}r)Y_n'(\sqrt{\lambda_k})-J_n'(\sqrt{\lambda_k})Y_n(\sqrt{\lambda_k}r)$$ and $\lambda_k$ must satisfies (except for $\lambda_k=0$ when $n=0$: 
\begin{equation}\label{bessel investigation}
J_n'(\sqrt{\lambda_k}\delta)Y_n'(\sqrt{\lambda_k})-J_n'(\sqrt{\lambda_k})Y_n'(\sqrt{\lambda_k}\delta)=0.
\end{equation}
And we have the following lemmas regarding the distribution of eigenvalues as $\delta$ goes to 1. 
\begin{lemma}
	\label{asymptotic delta}
	For any fixed $n$, the solution $\lambda_k$ to \eqref{radial eigenvalue equation} can be represented as $\lambda_{n,j}$, where $j=1,2,3,...$ and $$0\leq\lambda_{n,1}<\lambda_{n,2}<\lambda_{n,3}<...$$ and $\lambda_{n,j}\rightarrow \infty$ as $j\rightarrow \infty$. While for each fixed $j$, we have $$\lambda_{n,j}\rightarrow 
	\begin{cases}
	n^2,&\text{if}\ j=1,\\
	\infty,&\text{if}\ j>1,
	\end{cases}$$ as $\delta\rightarrow 1$. Also $\lambda_{n,1}>n^2/\delta^2$ for any $\delta>1$.
\end{lemma}
\begin{proof}
	$\lambda_{n,1}<\lambda_{n,2}<\lambda_{n,3}<...$ comes naturally from the Sturm-Liouville theory. \\
	For the second part, let's prove the following claims:
	\begin{enumerate}
		\item If we denote the smallest positive zeros of $J_n,J_n',Y_n,Y_n'$ as $j_n,j_n',y_n,y_n'$, then we have $n<j_n'<y_n'$.
		\item $J_n'$ and $Y_n'$ have alternating zeros which go to infinity when $x>n$.
	\end{enumerate}
	Proof of claim 1: \cite{watson1995treatise} page 485-487 has shown that $n<j'_n<y_n<j_n$ and $J_n^2+Y_n^2$ is decreasing on the interval $0<x<j_n'$. Hence $$J_n J_n'+Y_n Y_n'\leq 0$$ on this interval. Since $J_n,J_n'>0$ and $Y_n<0$ on this interval, $Y_n'>0$ in the interval. Furthermore, $$x(x Y_n')'=(n^2-x^2)Y_n$$ holds, using $Y_n'(n)>0$ and $Y_n<0$ for $n<x<j_n'$, we derive that $Y_n'(j_n')>0$ hence $n<j_n'<y_n'$.\\
	Proof of Claim 2: We know $J_n$ and $Y_n$ have alternating zeros that go to infinity, in between each consecutive zeros of $J_n$ or $Y_n$ there must be a zero for $J_n'$ or $Y_n'$ hence all the zeros must go to infinity. Furthermore it's straightforward to verify that $Y_n'$ and $J_n'$ satififies:
	$$T''+\frac{x^2-3n^2}{x(x^2-n^2)}T'+\left[1-\frac{1+n^2}{x^2}+\frac{2n^2}{x^2(n^2-x^2)}\right]T=0.$$ Since $Y_n'$ and $J_n'$ has no common zeros (a property from linear independence between $(Y_n, Y_n')$ and $(J_n, J_n')$, $(Y_n', Y_n'')$ and $(J_n', J_n'')$ are independent as well. Hence by the Sturm-Liouville theory and claim 1, they have alternating zeros and all zeros are greater than $n$.\\
	Proof of lemma: We investigate the following quantity: 
	\begin{equation}
	\begin{aligned}
	\left(\frac{J_n'(x)}{Y_n'(x)}\right)'&=\frac{-Y_n''J_n'+J_n''Y_n'}{(Y_n')^2}\\
	&=\frac{1-\frac{n^2}{x^2}}{(Y_n')^2}(Y_n J_n'- J_n Y_n')\\
	&=\frac{2(n^2-x^2)}{\pi x^3 (Y_n')^2}.
	\end{aligned}
	\end{equation}
	This shows, along with claim 1, that the function $\frac{J_n'(x)}{Y_n'(x)}$ first increase and peak at $x=n$ and then decreases to $-\infty$ at $y_n'$. And from $y_n'$ on, by claim 2, this function decrease from $+\infty$ to $-\infty$ at each interval between consecutive zeros of $Y_n'$. \\
	Now the equation \eqref{bessel investigation} can be rewritten as 
	$$\frac{J_n'(\sqrt{\lambda_k})}{Y_n'(\sqrt{\lambda_k})}=\frac{J_n'(\delta \sqrt{\lambda_k})}{Y_n'(\delta \sqrt{\lambda_k})},$$ 
	it then follows naturally that when $\delta$ is close to $1$, $\lambda_{n,1}$ lies in $(n^2/\delta^2,n^2)$ for $n\geq 1$ while $\lambda_{0,1}=0$, hence $\lambda_{n,1}\rightarrow n^2$ when $\delta\rightarrow 1$. While $(\delta-1)\sqrt{\lambda_{n,2}}$ is always greater than the minimum of distance between any two different zeros of $Y_n'$ and $J_n'$ that are smaller than $\delta \sqrt{\lambda_{n,2}}$, this forces $\sqrt{\lambda_{n,2}}$ to go to $\infty$ as $\delta\rightarrow 1$, and the case for any $\lambda_{n,j},j>1$ follows true by the first part of the lemma.
	
\end{proof} 

\begin{hypothesis}
	$\lambda_{n,j}\neq \lambda_{m,k}$ for any $\{n,j\}\neq\{m,k\}$.
\end{hypothesis}
\begin{remark}
	The above hypothesis resembles that of Bourget's hypothesis regarding zeros of Bessel function of the first kind, which was proved to be true  \cite{siegel2014einige}\cite{ashu2013some}. Here the difference is our hypothesis focuses on the zeros of cross products involving the derivatives of the first and the second Bessel functions. 
\end{remark}
Given the above lemma and hypothesis, we can see that the multiplicity of $\lambda_k$ in \eqref{radial eigenvalue equation} is two when $n\geq 1$, and is one if $n=0$. The hypothesis is not needed in the limiting case when $\delta$ is close to one. 

Then the eigenvalue problem \eqref{ep} can be solved using the ansatz \begin{align}\label{ansatz}
e_{k,l}=\begin{pmatrix}
u\\
v\\
\end{pmatrix}=
\begin{pmatrix}
u_{k,l} B_k\\
v_{k,l} B_k 
\end{pmatrix}, \text{ where } l\in \{1,2\}.
\end{align}

Substituting \eqref{ansatz} into the equation \eqref{ep}
, we obtain that 
\begin{align}
\label{uvrelation}
\left\{\begin{array}{ll}
~-\lambda_k u_k+R^2\Big[\frac{\lambda-a}{a+\lambda}u_k
+(a+\lambda)^2v_k\Big]
=\beta(\lambda) u_k,\\
-d\lambda_k v_k-R^2\Big[\frac{2\lambda}{a+\lambda}u_k
+(a+\lambda)^2v_k\Big]=\beta(\lambda) v_k.
\end{array}\right.
\end{align}
From this we can see that $\beta$ satisfies:
\begin{equation}\label{eigenvalue equation}
(\beta-R^2\frac{\lambda-a}{\lambda +a}+\lambda_k)(\beta+R^2(a+\lambda)^2+d\lambda_k)=-2R^4\lambda(a+\lambda).
\end{equation}
Then 

\begin{align}
\beta_{k,1\text{ or } 2}=\frac{1}{2}\Bigg\{&-\Big[\lambda_k-R^2\frac{\lambda-a}{a+\lambda}+d\lambda_k+R^2(a+\lambda)^2
\Big]\\
&\pm\Big[\big(\lambda_k-R^2\frac{\lambda-a}{a+\lambda}-d\lambda_k-R^2(a+\lambda)^2\big)^2-8R^4\lambda(a+\lambda)\Big]^{\frac{1}{2}}\Bigg\}.
\end{align}

Now we proceed to classify the parameters $R,d,a,\lambda$ according to what unstable modes this linear equation generates while omitting the critical cases. We obtain a Proposition as follows:
\begin{theorem}
	\label{parameter thm}
	In what follows, we assume $R,d,a,\lambda$ are all positive numbers. Here $e_{k,l}$ is stable (unstable) means the corresponding eigenvalue $\beta_{k,l}$ has negative (positive) real  parts.
	\begin{enumerate}
		\item If $(\lambda-a)>(\lambda+a)^3$, then the constant eigenvector $e_{0,l}$ is unstable.
		\item If $(\lambda-a)<(\lambda+a)^3$ and $d(\lambda-a)>(a+\lambda)^3+2\sqrt{d}(a+\lambda)^2$, then $e_{k,2}$ is unstable, where $\lambda_k$ lies between:
		\begin{equation}
		\frac{R^2}{2d}\left[d\frac{\lambda-a}{\lambda+a}-(\lambda+a)^2\pm\sqrt{\left[d\frac{\lambda-a}{\lambda+a}-(\lambda+a)^2\right]^2-4d(\lambda+a)^2} \right]
		\end{equation} 
		and all the other eigenvectors (including $e_{0,l}$) are stable. If no $\lambda_k$ lies in this interval then every eigenvector is stable. 
		\item If $(\lambda-a)<(\lambda+a)^3$ and $d(\lambda-a)<(a+\lambda)^3+2\sqrt{d}(a+\lambda)$, then all the eigenvectors are stable. 
	\end{enumerate}
\end{theorem} 
\begin{proof} 
	\textbf{Case 1:}\\
	The equation \eqref{eigenvalue equation} can be written as:
	\begin{equation}\label{eigenvalue equation expanded}
	\begin{aligned}
	&\beta^2+\{\lambda_k(d+1)-R^2[\frac{\lambda-a}{\lambda+a}-(a+\lambda)^2]\}\beta\\&+d\lambda_k^2+R^2[(a+\lambda)^2-d\frac{\lambda-a}{\lambda+a}]\lambda_k+R^4(a+\lambda)^2=0.
	\end{aligned}
	\end{equation} It's easy to see if $\lambda_k=0$ and $(\lambda-a)>(\lambda+a)^3$ the equation has either two positive roots or two complex roots with positive real parts hence the conclusion. \\
	\textbf{Case 2:} The discriminant for the quadratic function $$d\lambda_k^2+R^2[(a+\lambda)^2-d\frac{\lambda-a}{\lambda+a}]\lambda_k+R^4(a+\lambda)^2$$ viewed as a function of $\lambda_k$ is $$R^4\left[d\frac{\lambda-a}{\lambda+a}-(\lambda+a)^2\right]^2-4d(\lambda+a)^2,$$ hence the second inequality appeared in case 2 guarantees the positiveness of the discriminant, since it also guarantees the coefficient for $\lambda_k$ in the quadratic function to be negative, basic quadratic analysis shows the for $\lambda_k$ between the two (positive) roots appeared in case 2, the constant term in \eqref{eigenvalue equation expanded} is negative hence guarantees a positive real root for \eqref{eigenvalue equation expanded}.\\
	\textbf{Case 3:} The second inequality in case 3 shows the discriminant above is negative hence the constant term in \eqref{eigenvalue equation expanded} is always positive hence no positive real part is possible for any of the modes. 
	
\end{proof}
The parameter space in the case 2 of the above theorem can be shown between the two surfaces (outside the small cylinder and inside the bigger cone) in figure \ref{parameter space}. Case 1 indicates the region inside the cylinder, while case 3 indicates the region outside both the cylinder and the ucone. 

\begin{figure}[htp]
	\centering
	\includegraphics[width=0.9\textwidth]{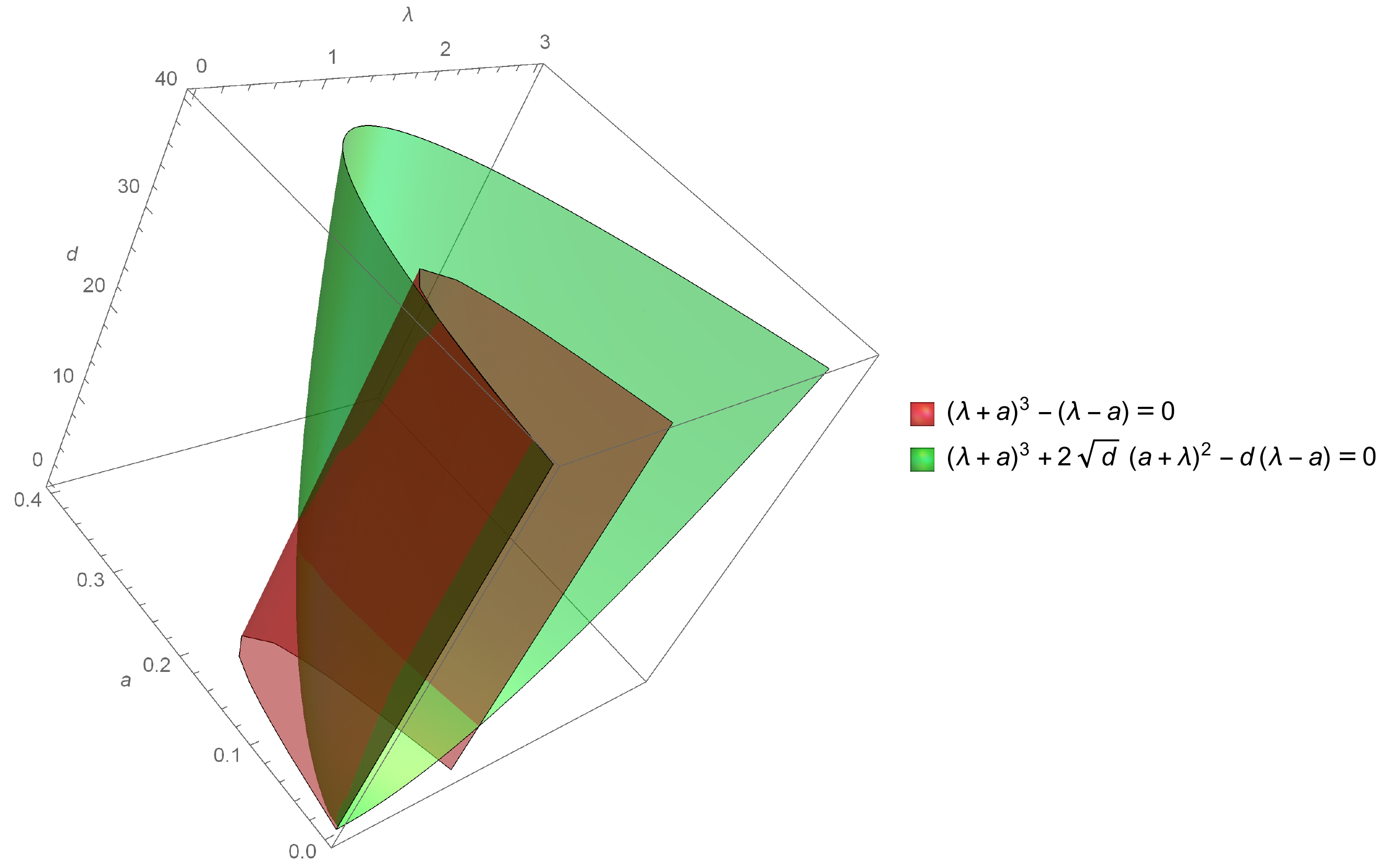}
	\caption{Classification in the parameter space}\label{parameter space}
\end{figure}
By treating $\lambda$ as a changing parameter, which represents the concentration of calcium ions, we can expect that when $\lambda$ increases from region 3 to region 2 in the above theorem, we should have a dynamic transition from a uniform steady state to a spatially variant steady state. We are particularly interested in this case because the other case where region 3 transitioned into region 1 is locally equivalent to an ODE version of \eqref{ep} without the diffusion terms, since all spatial variant modes are stable during the transition. Actually we can prove: 
\begin{lemma}
	If $\Re \beta_{k,l}<0$ for any $k>0$ and $l=1 \text{ or } 2$, than the constant function space $\{U|\nabla U=0 \text{ in } \Omega\}$ is exponentially attracting in a neighborhood of $\mathbf{0}$ in $X^{1/2}$ in equation \eqref{231}. 
\end{lemma}
\begin{proof}
	This is a direct consequence of the Theorem 6.1.4 of \cite{henry2006geometric}. With our invariant manifold being the constant function space.
\end{proof}
For simplicity, we assume $a>\frac{1}{3\sqrt{3}}$. It's easy to verify under this condition, $(\lambda-a)<(\lambda+a)^3$ always hold for any $\lambda>0$. Another simplification we will make is to restrict our transition purely inside region 2 instead of on the boundary of region 2. The reason behind this is the length of window for $\lambda_k$ is close to $0$ near boundary and it is not generic that there is a $\lambda_k$ lies precisely inside this small window. The third simplification we shall make is there is only one (not counting multiplicity) $\lambda_k$ entering the length window when $\lambda$ increase to a critical value. It's easy to see this is the generic case. The double entering case corresponds to boundary lines in Figure \ref{nc plot} later in numerical section hence is negligible. 

And we assert the PES condition under these restrictions as follows: 
\begin{theorem}
	\label{pes theorem}
	Assume $a>\frac{1}{3\sqrt{3}}$ and $d(\lambda-a)>(a+\lambda)^3+2\sqrt{d}(a+\lambda)^2$, if the positive two roots of $d(\lambda-a)=(a+\lambda)^3+2\sqrt{d}(a+\lambda)^2$ are $\lambda_{min}$ and $\lambda_{max}$, we further require $\frac{R^2}{2d}(a+\lambda_{min})\notin L=\{\lambda_k,k\in \mathbb{N}\}$ and  $n(I_\lambda\cap L)=1$ when $\lambda\in(\min \{\lambda, I_\lambda\cap L\neq \emptyset\},\min \{\lambda, I_\lambda\cap L\neq \emptyset\}+\epsilon)$ for a certain $\epsilon>0$ where \begin{equation}
	\label{intervallambda}
	\begin{split}I_\lambda=&\left(\frac{R^2}{2d}\left[d\frac{\lambda-a}{\lambda+a}-(\lambda+a)^2-\sqrt{\left[d\frac{\lambda-a}{\lambda+a}-(\lambda+a)^2\right]^2-4d(\lambda+a)^2} \right]\right.,\\
	&\left.\frac{R^2}{2d}\left[d\frac{\lambda-a}{\lambda+a}-(\lambda+a)^2+\sqrt{\left[d\frac{\lambda-a}{\lambda+a}-(\lambda+a)^2\right]^2-4d(\lambda+a)^2} \right]\right).
	\end{split}
	\end{equation}
	we have the following: (we treat $\lambda$ as our parameter hence the eigenvalues can be written as $\beta_{k,l}(\lambda)$)
	\begin{align}\label{pes}
	\begin{cases}
	\beta_{k_c,2}(\lambda)
	\begin{cases}
	>0,\lambda_c<\lambda<\lambda_c+\epsilon\\
	=0,\quad \lambda=\lambda_c\\
	<0,\quad 0<\lambda<\lambda_c
	\end{cases},\\
	\Re\beta_{m,l}<0,\beta_{k_c,1}<0,\text{ where }m\neq k_c, 0<\lambda<\lambda_c+\epsilon.
	\end{cases}
	\end{align}
	Here $\lambda_c=\min \{\lambda, I_\lambda\cap L\neq \emptyset\}$ and $k_c$ corresponding to the only $\lambda_{k_c}$ that is inside $I_\lambda\cap L$ when $\lambda_c<\lambda<\lambda_c+\epsilon$.
\end{theorem}
\begin{remark}
	\label{remark on multiplicity}
	Notice in the above theorem the multiplicity of $\beta_{k_c,2}$ depends on the underlying $\lambda_{k_c}$. If the corresponding $n_c$ (defined to be the $n$ in the critical eigenvector $R(r)\cos{n\theta}$ or $R(r)\sin{n\theta}$) is $0$ then the multiplicity is $1$, if $n_c>0$ then the multiplicity is $2$. When $\delta\rightarrow 1$, according to Lemma \ref{asymptotic delta} and boundedness of $I_\lambda$ when $\lambda$ satisfies the restrictions in Theorem \ref{pes theorem}, only multiplicity $2$ transition will occur. 
\end{remark}

\section{Dynamical transition theorems}
Suppose in this section the eigenvectors can be represented as 
\begin{align*}
e_{n j l c}=\begin{pmatrix}u_{n j l}\\v_{n j l}\end{pmatrix} \cos{n \theta} R_{n j}(r),\quad 
e_{n j l s}=\begin{pmatrix}u_{n j l}\\v_{n j l}\end{pmatrix} \sin{n \theta} R_{n j}(r),
\end{align*}
for $n>1,j\in \mathbb{Z}^+,l=1\text{ or }2$, and 
$$  e_{0 j l}=\begin{pmatrix}u_{0 j l}\\v_{0 j l}\end{pmatrix}R_{0 j}(r)$$ for $n=0$, and all the other quantities are denoted in the same fashion. (for example $\beta_{njl}$ to denote the eigenvalues). Here we assume all the eigenvectors are normalized, i.e. $\int _\Omega |e_{\text{sub}}|^2=1$ for any subscripts, and $e_{nj1c}=\overline{e_{nj2c}}$, $e_{nj1c}=\overline{e_{nj2c}}$ and $e_{0j1}=\overline{e_{0j2}}$ for complex eigenvectors, otherwise all eigenvectors are real. 
Moreover we can standardize the choice of $u_{njl}$ and $v_{njl}$ so that they have the following properties given the restrictions in the PES:
\begin{equation}
\begin{aligned}
&u_{nj1},u_{nj2}>0, u_{nj1}=u_{nj2}\text{ when } v\text{ is not real,}\\
&\Re v_{njl}<0,\\
&\pi  (|u_{njl}|^2+|v_{njl}|^2)\int_1^\delta R_{nj}^2 r\mathrm{d}r=1.
\end{aligned}
\end{equation} 
The only non-trivial part of the above properties is the second one, which can be derived from \eqref{uvrelation}, namely, 
\begin{equation*}
\frac{-\lambda_{nj}u_{nj}+R^2\left[\frac{\lambda-a}{\lambda+a}u_{nj}+(a+\lambda)^2 v_{nj}\right]}{u_{nj}}
=\frac{-d\lambda_{nj}v_{nj}-R^2\left[\frac{2\lambda}{\lambda+a}u_{nj}+(a+\lambda)^2 v_{nj}\right]}{v_{nj}}.
\end{equation*}
Hence if we denote $p_l=v_{njl}/u_{njl}$, then it's easy to derive
\begin{equation*}
\begin{aligned}
p_1 \cdot p_2&= \frac{2\lambda}{(a+\lambda)^3}>0,\\
p_1+p_2&=-\frac{1}{R^2(a+\lambda)^2}\left[R^2\frac{\lambda-a}{\lambda+a}+R^2(a+\lambda)^2+(d-1)\lambda_{nj}\right]<0,
\end{aligned}
\end{equation*}
then the properties follows. See Theorem 2.1.3 of \cite{ma2014phase} for a classification of dynamical transition into three types. 

Here we first define the following quantity $q(\lambda)$:
\begin{equation*}
\begin{aligned}
q(\lambda)=&\sum_{j\in \mathbb{Z}^+, l\in{1,2}}(B_{0jl}+\frac{1}{4}B_{2n_c j l})+\frac{3\pi}{4}\gamma_3u_{n_c j_c 2}^2 v_{n_c j_c 2} \int_1^\delta R_{n_c j_c}^4(r) r\mathrm{d}r,\\
B_{njl}=&-\frac{\pi^2}{\beta_{njl}}(\overline{u_{njl}}-\overline{v_{njl}})(2\gamma_2 u_{n_c j_c 2}u_{njl}+\gamma_{12}u_{n_c j_c 2}v_{njl}+\gamma_{12} v_{n_c j_c 2} u_{njl})\cdot\\
&(\gamma_2 u_{n_c j_c l}^2+\gamma_{12} u_{n_c j_c 2} v_{n_c j_c 2})(\int_1^\delta R_{n_c j_c}^2 R_{nj} r\mathrm{d}r)^2.
\end{aligned}
\end{equation*}
 Recall $n_c$ is defined to be the $n$ in the critical eigenvector $R(r)\cos{n\theta}$ or $R(r)\sin{n\theta}$. Notice that $B_{nj1}=\overline{B_{nj2}}$ for non-real eigenvectors hence $q(\lambda)$ is a real number. 
\begin{theorem}
	\label{transition theorem 2d}
	Under the conditions in Theorem \ref{pes theorem} and assume $n_c\neq 0$, then we have:
	\begin{enumerate}
		\item When $q(\lambda_c)<0$, the problem \eqref{231} undergoes a type I (continuous) transition near $\lambda_c$. And the system bifurcates from $0$ on $\lambda>\lambda_c$ to an attractor that's homologically equivalent to a circle. 
		\item When $q(\lambda_c)>0$, the problem \eqref{231} undergoes a type II
		(Catastrophic) transition near $\lambda_c$.
	\end{enumerate}
	
\end{theorem}
\begin{proof}
	Suppose the crossing eigenvectors are:
	\begin{align*}
	e_c=\begin{pmatrix}u_{n_c j_c 2}\\v_{n_c j_c 2}\end{pmatrix} \cos{n_c \theta} R_{n_c j_c}(r),\quad 
	e_s=\begin{pmatrix}u_{n_c j_c 2}\\v_{n_c j_c 2}\end{pmatrix} \sin{n_c \theta} R_{n_c j_c}(r). 
	\end{align*}
	
	By noting the existence of a center manifold in the neighborhood of origin, the projection of the center manifold onto the space $\{y_c e_c+y_s e_s\}$ can be written as follows:
	\begin{equation}
	\label{reduced raw}
	\begin{aligned}
	\frac{\mathrm{d}y_c}{\mathrm{d}t}&=\beta_{n_c j_c 2}y_c+
	\int_\Omega G(y_c e_c+y_s e_s+\Phi(y_c,y_s))\cdot e_c\\
	\frac{\mathrm{d}y_s}{\mathrm{d}t}&=\beta_{n_c j_c 2}y_s+
	\int_\Omega G(y_c e_c+y_s e_s+\Phi(y_c,y_s))\cdot e_s,
	\end{aligned}
	\end{equation}
	where $\Phi$ is the local center manifold function from space $\{y_c e_c+y_s e_s\}$ to its complement in $X$. This equation will describe the same dynamic as the center manifold locally around $0$.
	According to Theorem A.1.1 of \cite{ma2014phase}, this center manifold can be approximated as follows:
	\[
	-L_\lambda\Phi(y_c,y_s)=P_2 G_2 (y_c,y_s)+o(2),
	\]
	where $P_2$ is the standard projection onto the complement space, $G_2$ is the quadratic part of $G$, and $o(2)=o(|y_c|^2+|y_s|^2)+O(|\beta_{n_c j_c 2}|(|y_c|^2+|y_s|^2))$.
	
	To get a third order approximation of \eqref{reduced raw}, we only need second order terms in $\Phi$, which is equal to $\Phi_2(y_c,y_s)=-L_\lambda^{-1}P_2G_2(y_c,y_s)$ by the above equation. It will be clear later that the only useful modes in $\Phi_2$ are of follows:
	\begin{enumerate}
		\item Modes $e_{0jl}$ for $j\in \mathbb{Z}^+$ and $l=1,2$ and 
		\begin{equation}
		\label{phi1}
		\begin{aligned}
		\int_\Omega\Phi_2(y_c,y_s)\cdot\overline{e_{0jl}}
		=&-\frac{1}{\beta_{0jl}}\int_\Omega G_2(y_c,y_s)\cdot \overline{e_{0jl}}\\
		=&-\frac{1}{\beta_{0jl}}\int_\Omega
		(\gamma_2 u_{n_c j_c 2}^2+\gamma_{12}u_{n_c j_c 2}v_{n_c j_c 2})(\overline{u_{0jl}}-\overline{v_{0jl}})\\
		&\cdot R_{n_c j_c}^2 R_{0j}(y_c \cos{n_c \theta}+y_s\sin{n_c\theta})^2\\
		=&-\frac{\pi}{\beta_{0jl}}
		(\gamma_2 u_{n_c j_c 2}^2+\gamma_{12}u_{n_c j_c 2}v_{n_c j_c 2})(\overline{u_{0jl}}-\overline{v_{0jl}})\\
		&\cdot \int_1^\delta R_{n_c j_c}^2 R_{0j}r\mathrm{d}r \ (y_c^2+y_s^2).
		\end{aligned}
		\end{equation}
		\item Modes $e_{2n_c jlc}$ and $e_{2n_c jls}$ for $j\in \mathbb{Z}^+$ and $l=1,2$. Using similiar calculation as above we obtain:
		\begin{equation}
		\label{phi2}
		\begin{aligned}
		\int_\Omega\Phi_2(y_c,y_s)\cdot\overline{e_{2n_c jlc}}
		=&-\frac{\pi}{\beta_{2n_cjl}}
		(\gamma_2 u_{n_c j_c 2}^2+\gamma_{12}u_{n_c j_c 2}v_{n_c j_c 2})(\overline{u_{2n_cjl}}-\overline{v_{2n_cjl}})\\
		&\cdot \int_1^\delta R_{n_c j_c}^2 R_{2n_cj}r\mathrm{d}r \ \frac{y_c^2-y_s^2}{2},
		\end{aligned}
		\end{equation}
		and
		\begin{equation}
		\label{phi3}
		\begin{aligned}
		\int_\Omega\Phi_2(y_c,y_s)\cdot\overline{e_{2n_c jls}}
		=&-\frac{\pi}{\beta_{2n_cjl}}
		(\gamma_2 u_{n_c j_c 2}^2+\gamma_{12}u_{n_c j_c 2}v_{n_c j_c 2})(\overline{u_{2n_cjl}}-\overline{v_{2n_cjl}})\\
		&\cdot \int_1^\delta R_{n_c j_c}^2 R_{2n_cj}r\mathrm{d}r \ y_c y_s.
		\end{aligned}
		\end{equation}
	\end{enumerate}
	Now we go back to the equation \eqref{reduced raw}, calculation shows
	\begin{equation}
	\begin{aligned}
	&\int_\Omega G(y_c e_c+y_s e_s+\Phi(y_c,y_s))\cdot e_c\\
	=&\int_\Omega\big\{\gamma_2[R_{n_c j_c}(y_c \cos{n_c\theta}+y_s\sin{n_c\theta})]^2 u_{n_c j_c 2}^2\\
	&+\gamma_{12}[R_{n_c j_c}(y_c \cos{n_c\theta}+y_s\sin{n_c\theta})]^2 u_{n_c j_c 2}v_{n_c j_c 2}\\
	&+\gamma_3[R_{n_c j_c}(y_c \cos{n_c\theta}+y_s\sin{n_c\theta})]^3 u_{n_c j_c 2}^2 v_{n_c j_c 2}\\
	&+2\gamma_2[R_{n_c j_c}(y_c \cos{n_c\theta}+y_s\sin{n_c\theta})] u_{n_c j_c 2}\Phi_u (y_c,y_s)\\
	&+\gamma_{12}[R_{n_c j_c}(y_c \cos{n_c\theta}+y_s\sin{n_c\theta})] u_{n_c j_c 2}\Phi_v(y_c,y_s)\\
	&+\gamma_{12}[R_{n_c j_c}(y_c \cos{n_c\theta}+y_s\sin{n_c\theta})] v_{n_c j_c 2}\Phi_u(y_c,y_s)\big\}\\
	&\cdot \cos{n_c\theta}R_{n_c j_c}(u_{n_c j_c 2}-v_{n_c j_c 2})+o(3),
	\end{aligned}
	\end{equation}
	where $\Phi_u$ and $\Phi_v$ are $u$ and $v$ components of $\Phi_2$, and $o(3)$ is same notation as $o(2)$. It's clear from this equation that only those modes with $2n_c$ and $0$ in $\Phi_2$ wavenumbers are nonzero after integrating. After plug in with \eqref{phi1}, \eqref{phi2} and \eqref{phi3} and notice: 
	\begin{equation*}
	\begin{aligned}
	&\int_0^{2\pi}(y_c \cos{n_c\theta}+y_s\sin{n_c\theta})^2 \cos{n_c \theta}~\mathrm{d}\theta=0\\
	&\int_0^{2\pi}(y_c \cos{n_c\theta}+y_s\sin{n_c\theta})\cos{n_c \theta}~\mathrm{d}\theta=\pi y_c\\
	&\int_0^{2\pi}(y_c \cos{n_c\theta}+y_s\sin{n_c\theta})^3\cos{n_c\theta}~\mathrm{d}\theta=\frac{3\pi}{4}y_c(y_c^2+y_s^2)\\
	&\int_0^{2\pi}(y_c \cos{n_c\theta}+y_s\sin{n_c\theta})\cos{n_c\theta}\cos{2n_c\theta}~\mathrm{d}\theta=\frac{\pi}{2}y_c\\
	&\int_0^{2\pi}(y_c \cos{n_c\theta}+y_s\sin{n_c\theta})\cos{n_c\theta}\sin{2n_c\theta}~\mathrm{d}\theta=\frac{\pi}{2}y_s
	\end{aligned}
	\end{equation*}
	
	we obtain: 
	\begin{equation}
	\label{reduced1}
	\begin{aligned}
	&\int_\Omega G(y_c e_c+y_s e_s+\Phi(y_c,y_s))\cdot e_c\\
	=&\Big[\sum_{j\in \mathbb{Z}^+, l\in{1,2}}(B_{0jl}+\frac{1}{4}B_{2n_c j l})+\frac{3\pi}{4}\gamma_3u_{n_c j_c 2}^2 v_{n_c j_c 2} \int_1^\delta R_{n_c j_c}^4(r) r\mathrm{d}r\Big]\\
	&\cdot(u_{n_c j_c 2}-v_{n_c j_c 2})y_c(y_c^2+y_s^2)+o(3),
	\end{aligned}
	\end{equation}
	where
	\begin{equation*}
	\begin{aligned}
	B_{njl}=&-\frac{\pi^2}{\beta_{njl}}(\overline{u_{njl}}-\overline{v_{njl}})(2\gamma_2 u_{n_c j_c 2}u_{njl}+\gamma_{12}u_{n_c j_c 2}v_{njl}+\gamma_{12} v_{n_c j_c 2} u_{njl})\\
	&\cdot (\gamma_2 u_{n_c j_c l}^2+\gamma_{12} u_{n_c j_c 2} v_{n_c j_c 2})(\int_1^\delta R_{n_c j_c}^2 R_{nj} r\mathrm{d}r)^2
	\end{aligned}
	\end{equation*}
	as is in the statement of the theorem.
	
	Hence the equation \eqref{reduced raw} can be rewritten as (by symmetry between $y_c$ and $y_s$)
	\begin{equation}
	\label{reduced final}
	\begin{aligned}
	\frac{\mathrm{d}y_c}{\mathrm{d}t}&=\beta_{n_c j_c 2}y_c+
	q(\lambda)(u_{n_c j_c 2}-v_{n_c j_c 2})y_c(y_c^2+y_s^2)+o(3)\\
	\frac{\mathrm{d}y_s}{\mathrm{d}t}&=\beta_{n_c j_c 2}y_s+
	q(\lambda)(u_{n_c j_c 2}-v_{n_c j_c 2})y_s(y_c^2+y_s^2)+o(3),
	\end{aligned}
	\end{equation}
	at $\beta_{n_c j_c 2}=0$, $o(3)$ reduce to $o(|y_c|^3+|y_s|^3)$ and energy estimate gives
	\[
	\frac{1}{2}\frac{\mathrm{d}|y|^2}{\mathrm{d}t}=q(\lambda_c)(u_{n_c j_c 2}-v_{n_c j_c 2})|y|^4+o(|y|^4).
	\]
	Since $u_{n_c j_c 2}-v_{n_c j_c 2}$ is always positive by previous discussion, the steady state $0$ is locally asymptotically stable at critical $\lambda$ if $q(\lambda_c)$ is negative, which by Theorem 2.2.11 of \cite{ma2014phase} indicates a type I (continuous) type transition for the original system. If $q(\lambda_c)>0$, then by Theorem 2.4.16 of \cite{ma2014phase}, the system undergoes a type II (catastrophic) transition. 
	
\end{proof}

\tikzset{->-/.style={decoration={
			markings,
			mark=at position #1 with {\arrow{>}}},postaction={decorate}}}

\begin{theorem}
	Under the condition of Theorem \ref{pes theorem}, then if $n_c=0$ and $(\gamma_2 u_c+\gamma_{12}v_c)\int_1^\delta R_c^3 r \mathrm{d}r\neq 0$, (see proof for the notation of $u_c, v_c$ and $R_c$), we have the following assertions:
	\begin{enumerate}
		\item Equation \eqref{231} has a type III (random) transition from $(0,\lambda_c)$. More precisely, there exists a neighborhood $U\subset X$ of $0$ such that $U$ is separated into two disjoint open sets $U_1^\lambda$ and $U_2^\lambda$ by the stable manifold $\Gamma_\lambda$ of $u=0$ such that the local transition structure is as shown below and the following hold:
		\begin{enumerate}
			\item[a.] $U=U_1^\lambda +U_2^\lambda +\Gamma_\lambda$;
			\item[b.] the transition in $U_1^\lambda$ is a type II (catastrophic) transition;
			\item[c.] the transition in $U_2^\lambda$ is continuous. 
		\end{enumerate}
		\begin{tikzpicture}
		\filldraw [->-=.5] (0,0) circle (1pt) --node[below, near start]{$v^\lambda$} (1.5,0) ;
		\draw [->-=.5] (2.5,0) --node[below]{$u=0$} (1.5,0);
		\draw [->-=.5] (0,0) -- (-1,0);
		\draw [->-=.5] (0,1) -- (0,0);
		\draw [->-=.5] (0,-1) -- (0,0);
		\filldraw [->-=.5] (1.5,1) -- (1.5,0)circle(1pt);
		\draw [->-=.5] (1.5,-1) -- (1.5,0);
		\draw[line width=2pt,-stealth](2.8,0)--(3.4,0);
		\node at (0.75,-1.2){$\lambda<\lambda_c$};
		\end{tikzpicture}
		\begin{tikzpicture}
		\filldraw [->-=.5] (0,0)circle(1pt) -- (-1,0);
		\draw [->-=.5] (0,0) --node[below]{$u=0$} (1,0);
		\draw [->-=.5] (0,1) -- (0,0);
		\draw [->-=.5] (0,-1) -- (0,0);
		\draw[line width=2pt,-stealth](1.3,0)--(1.9,0);
		\node at (0,-1.2){$\lambda=\lambda_c$};
		\end{tikzpicture}
		\begin{tikzpicture}
		\filldraw [->-=.5] (0,0) circle(1pt)-- (1.5,0);
		\filldraw [->-=.5] (2.5,0) -- node[near end, below]{$v^\lambda$}(1.5,0)circle(1pt);
		\draw [->-=.5] (0,0)node[anchor=north west]{$u=0$} -- (-1,0);
		\draw [->-=.5] (0,1) -- (0,0);
		\draw [->-=.5] (0,-1) -- (0,0);
		\draw [->-=.5] (1.5,1) -- (1.5,0);
		\draw [->-=.5] (1.5,-1) -- (1.5,0);
		\node at (0.75,-1.2){$\lambda>\lambda_c$};
		\node at (0.75,0.5){$U_2^\lambda$};
		\node at (-0.5,0.5){$U_1^\lambda$};
		\end{tikzpicture}
		\item Equation \eqref{231} bifurcates in $U_2^\lambda$ to a unique singular point $v^\lambda$ on $\lambda>\lambda_c$ that is an attractor such that for every $\phi\in U_2^\lambda$, $\lim_{t\rightarrow\infty}||u(t,\phi)-v^\lambda||=0$.
		\item Equation \eqref{231} bifurcates on $\lambda<\lambda_c$ to a unique saddle point $v^\lambda$ with Morse index one.
		\item The bifurcated singular point $v^\lambda$ can be expressed as 
		\begin{equation*}
		v^\lambda=-\frac{\beta_{0j_c 2}(\lambda)}{2\pi (u_c-v_c)u_c(\gamma_2 u_c+\gamma_{12}v_c)\int_1^\delta R_c^3 r \mathrm{d}r}e+o(|\beta_{0j_c 2}|),
		\end{equation*}
	\end{enumerate}
\end{theorem}
\begin{proof}
	Here we assume the transitioning mode to be 
	$$e=\begin{pmatrix}u_{0j_c 2}\\v_{0j_c 2}
	\end{pmatrix}R_{0j_c}:=\begin{pmatrix}u_c\\v_c
	\end{pmatrix}R_c,$$then the projection of the center manifold is
	\begin{equation}
	\frac{\mathrm{d}y}{\mathrm{d}t}=\beta_{0j_c 2}y+\int_\Omega G(ye+\Phi(y))\cdot e,
	\end{equation}
	and 
	\begin{equation*}
	\begin{aligned}
	&G(ye+\Phi(y))\\
	=&G\left(yR_{0j_c}\binom{u_c}{v_c}+\binom{\Phi_u}{\Phi_v}\right)\\
	=&\binom{1}{-1}(\gamma_2 R_c^2 u_c^2 y^2+\gamma_{12}R_c^2 u_c v_c y^2)+o(2).
	\end{aligned}
	\end{equation*}
	Hence
	\begin{equation*}
	\begin{aligned}
	\frac{\mathrm{d}y}{\mathrm{d}t}&=\beta_{0j_c 2}y+\int_\Omega (u_c-v_c)(\gamma_2 u_c^2 y^2+\gamma_{12}u_c v_c y^2)R_c^3+o(2)\\
	&=\beta_{0j_c 2}y+2\pi (u_c-v_c)u_c(\gamma_2 u_c+\gamma_{12}v_c)\int_1^\delta R_c^3 r \mathrm{d}r~y^2+o(2),
	\end{aligned}
	\end{equation*}
	then the results follows directly from Theorem 2.3.2 of \cite{ma2014phase}.
\end{proof}

\section{Numerical Investigations}
As is indicated in Remark \ref{remark on multiplicity}, the transition when crossing eigenvalues are of multiplicity 2 is the most probable case when $\delta$ is close to $1$, we are more interested in Theorem \ref{transition theorem 2d}. However the coefficients $p(\lambda_c)$ have to be computed numerically. 

In this section we first show four plots of $n_c$ as a function of $d$ and $a$, with $R=4,8$ and $\delta=1.2$ and $2$. Then we calculate $p(\lambda_c)$ for several selected points, and then show the graph for several critical eigenvectors. All coding are done through Mathematica. We are not varying $R$ too much in our $n_c$ graph because the effect of $R$ on critical eigenvectors are quite clear from the formula of $I_\lambda$ in Theorem \ref{pes theorem}.

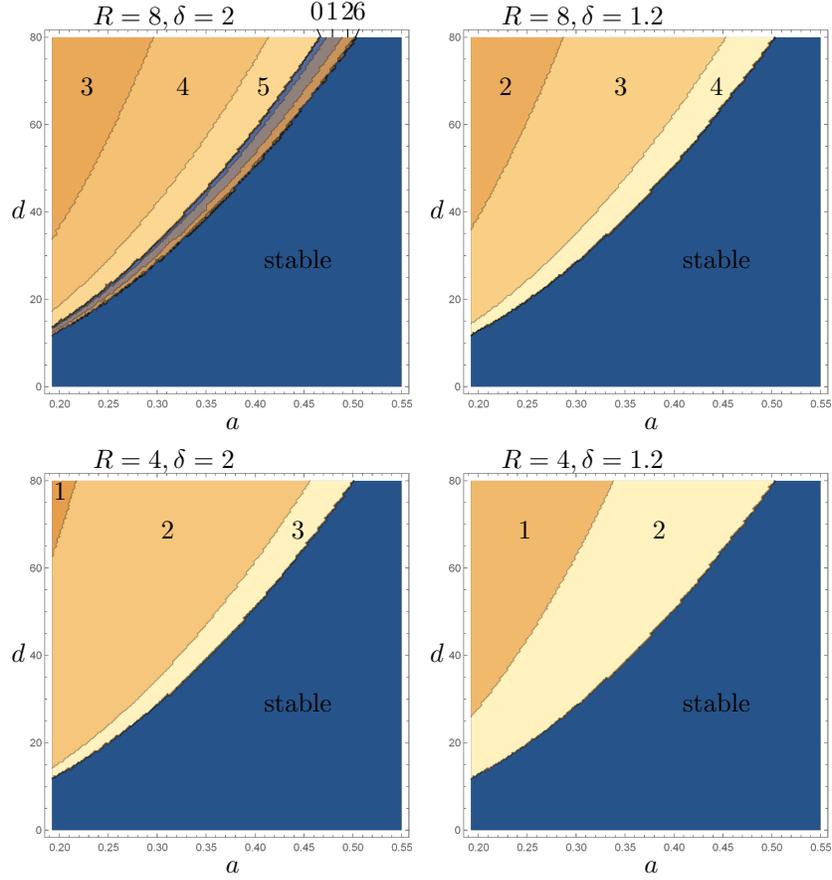
\begin{figure}[htp]
	\noindent
	\begin{tikzpicture}
	\node[anchor=south west,inner sep=0](image) at (0,0) {\includegraphics[width=2 in]{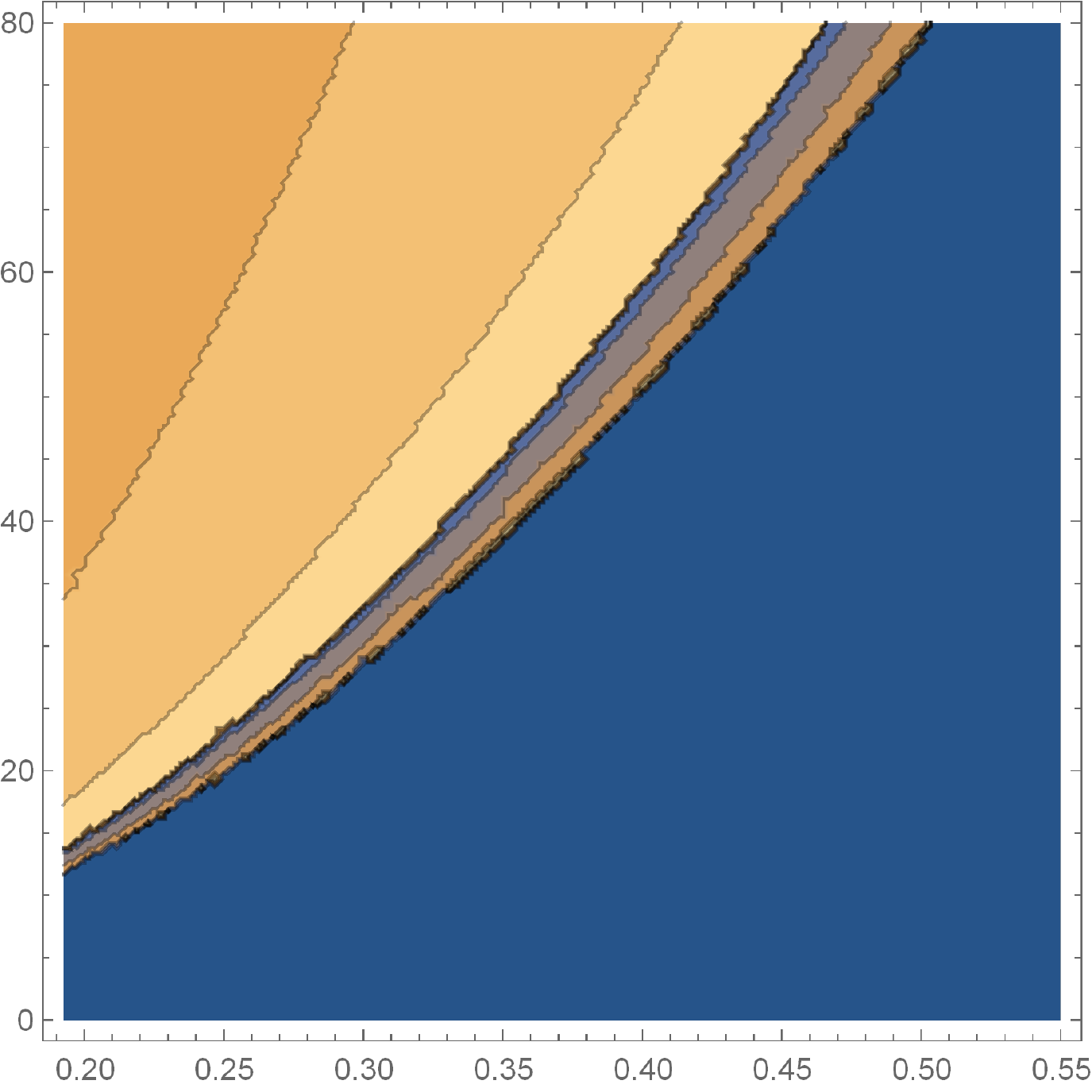}};
	\begin{scope}[x={(image.south east)},y={(image.north west)}]
	
	\node at(0.15,0.85){$3$};
	\node at(0.4,0.85){$4$};
	\node at(0.61,0.85){$5$};
	\node at(0.7,0.4){stable};
	\node at(0.53,-0.03){$a$};
	\node at(-0.03,0.53){$d$};
	\node at(0.35,1.03){$R=8,\delta=2$};
	\draw (0.76,0.978)--(0.75,1.)node[anchor=south]{$0$};
	\draw (0.79,0.978)--(0.79,1.)node[anchor=south]{$1$};
	\draw (0.83,0.978)--(0.83,1.)node[anchor=south]{$2$};
	\draw (0.85,0.978)--(0.86,1.)node[anchor=south]{$6$};
	\end{scope}
	\end{tikzpicture}
	\begin{tikzpicture}
	\node[anchor=south west,inner sep=0](image) at (0,0) {\includegraphics[width=2 in]{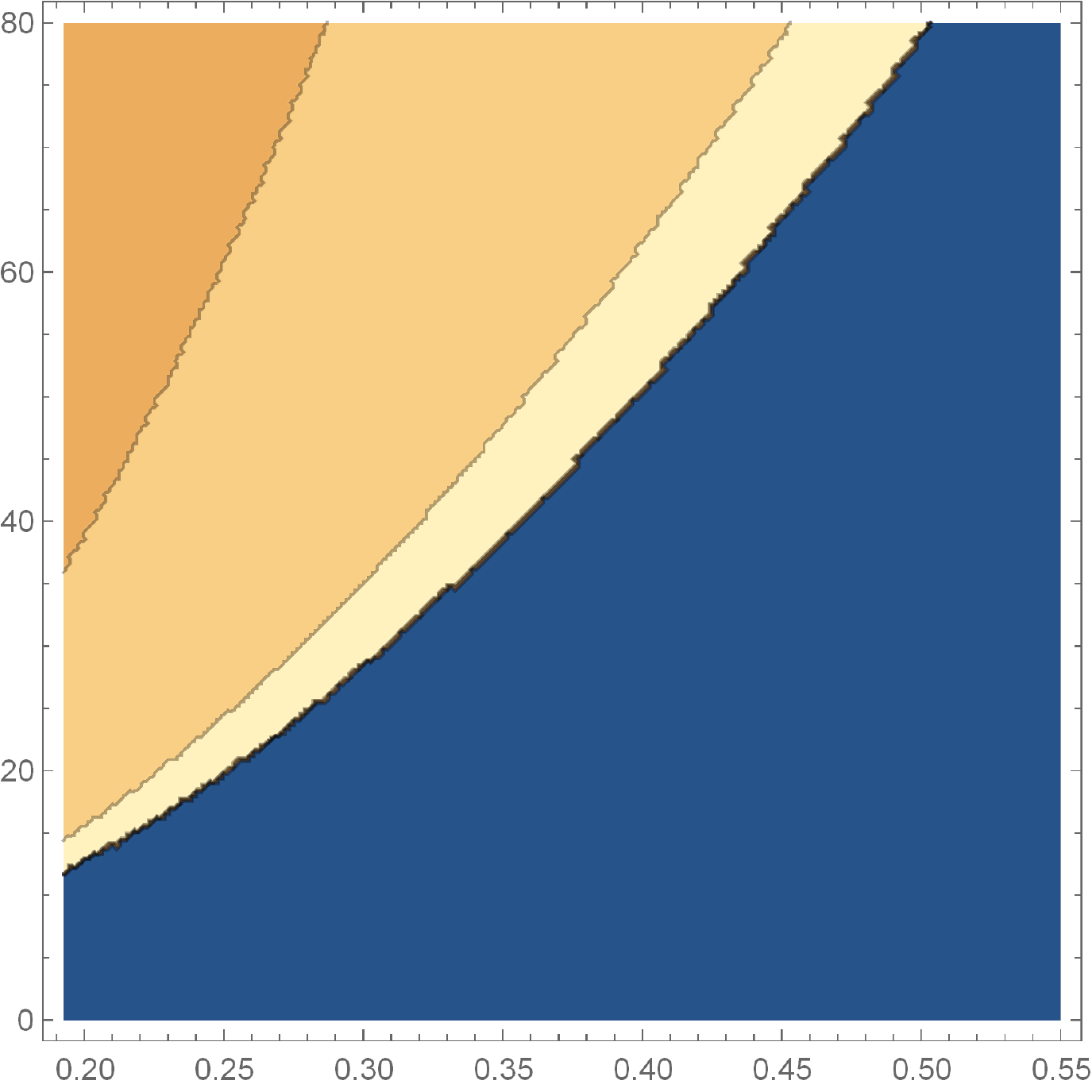}};
	\begin{scope}[x={(image.south east)},y={(image.north west)}]
	\node at(0.15,0.85){$2$};
	\node at(0.45,0.85){$3$};
	\node at(0.7,0.85){$4$};
	\node at(0.7,0.4){stable};
	\node at(0.53,-0.03){$a$};
	\node at(-0.03,0.53){$d$};
	\node at(0.35,1.03){$R=8,\delta=1.2$};
	\end{scope}
	\end{tikzpicture}
	
	\noindent
	\begin{tikzpicture}
	\node[anchor=south west,inner sep=0](image) at (0,0) {\includegraphics[width=2 in ]{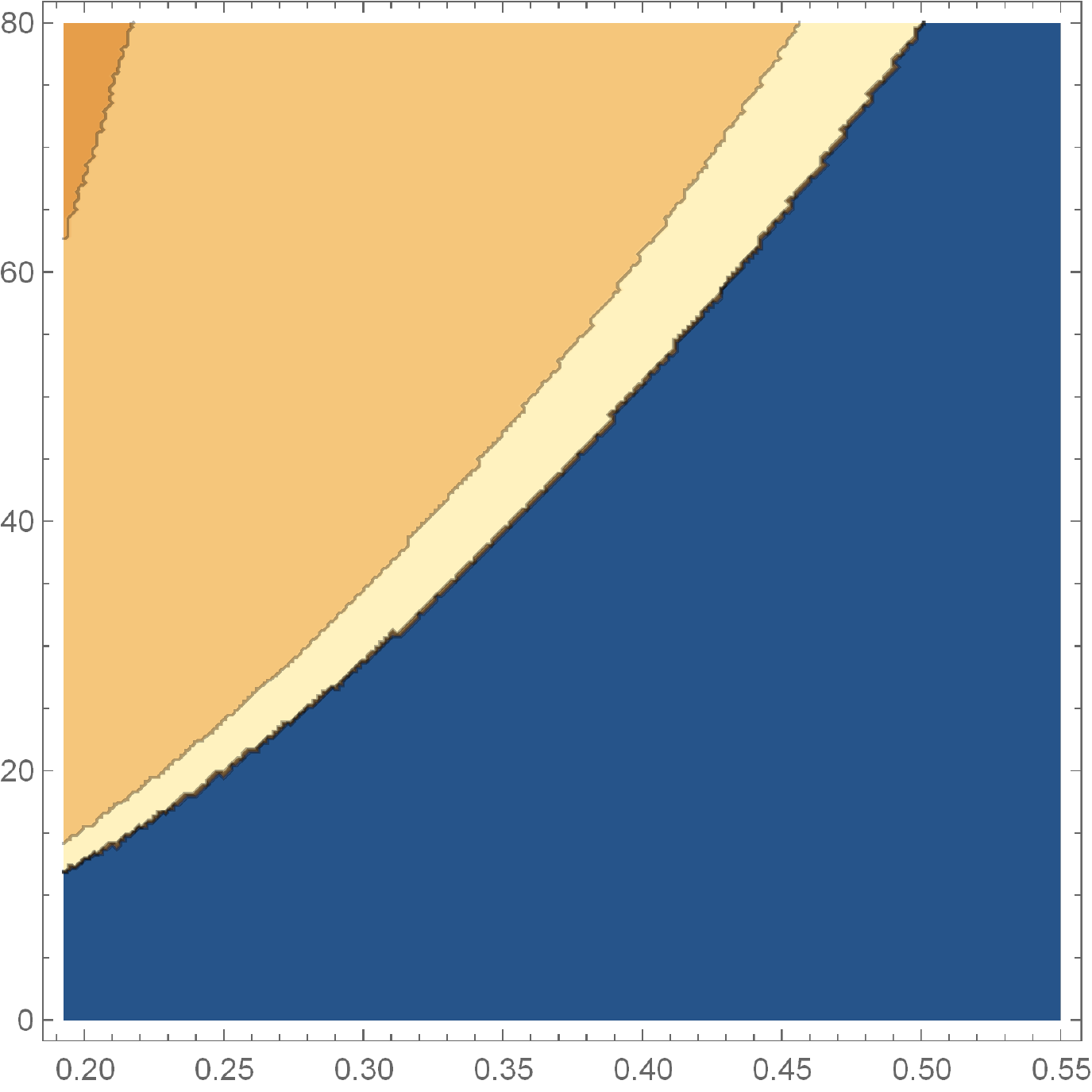}};
	\begin{scope}[x={(image.south east)},y={(image.north west)}]
	\node at(0.08,0.95){$1$};
	\node at(0.36,0.85){$2$};
	\node at(0.7,0.85){$3$};
	\node at(0.7,0.4){stable};
	\node at(0.53,-0.03){$a$};
	\node at(-0.03,0.53){$d$};
	\node at(0.35,1.03){$R=4,\delta=2$};
	\end{scope}
	\end{tikzpicture}
	\begin{tikzpicture}
	\node[anchor=south west,inner sep=0](image) at (0,0) {\includegraphics[width=2 in]{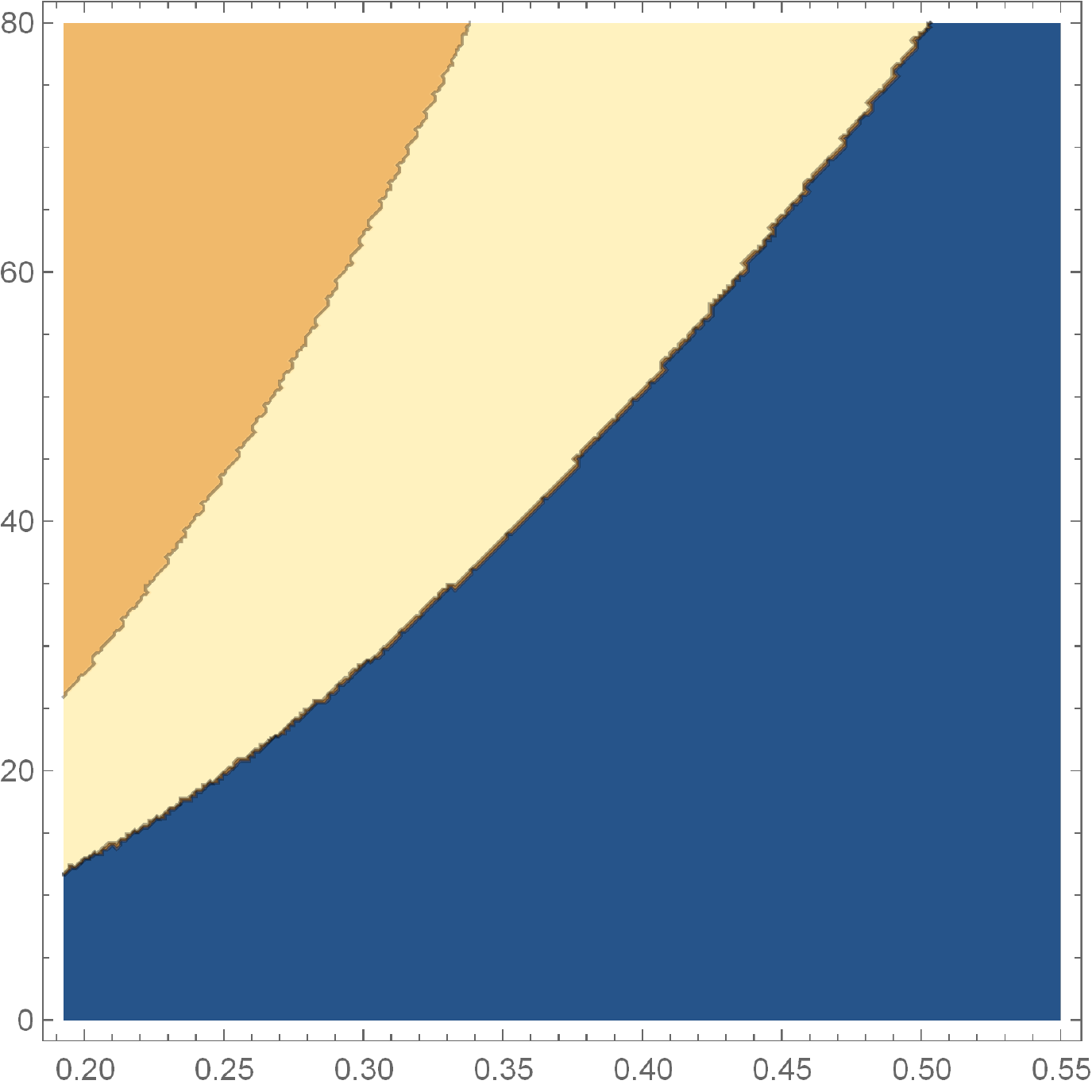}};
	\begin{scope}[x={(image.south east)},y={(image.north west)}]
	
	\node at(0.2,0.85){$1$};
	\node at(0.55,0.85){$2$};
	\node at(0.7,0.4){stable};
	\node at(0.53,-0.03){$a$};
	\node at(-0.03,0.53){$d$};
	\node at(0.35,1.03){$R=4,\delta=1.2$};
	\end{scope}
	\end{tikzpicture}
	\caption{$n_c$ plot against $a$ and $d$, numbers inside indicates critical wave number $n_c$ when $\lambda$ crosses $\lambda_c$}
	\label{nc plot}
\end{figure}

From the above graph we can see that critical wavenumber $n_c$ can change quite erratically when the parameters are near the boundary of the always stable regime. An interesting interpretation is the number of whorl hairs, which can be represented roughly as the number of positive patches of critical eigenfunctions, also behave irregularly. For example when $R=16,~\delta=2~d=80$, if we increase $a$ from $0.2$ to $0.55$, the number of whorl hairs during the initial whorl formation goes through $1,2,3,4,8,5,9,6,10,7,1,11,2,3,8,12,4,5,13$. This shows the importance of distribution of Bessel function zeros since critical eigenvectors are affected by them. 

The following is a chart for $p(\lambda_c)$ for various parameters when the critical wavenumber is not zero.
\begin{table}[htp]
	\centering
	\begin{tabular}{|c|c|c|c|c|c|}
		\hline
		$\delta $&$a$&$d$&$R$&$(n_c,j_c)$&$p(\lambda_c)(u_c-v_c)$ \\
		\hline     
		$1.05$&$0.2$&$13$&$4$&$(2,1)$ &$73.5557$ \\
		\hline     
		$1.05$&$0.2$&$15$&$4$&$(2,1)$ &$147.17$ \\
		\hline
		$1.05$&$0.2$&$20$&$4$&$(2,1)$&$240.462$\\
		\hline     
		$1.05$&$0.2$&$80$&$4$&$( 1,1)$ &$82.1464$ \\
		\hline     
		$1.05$&$0.4$&$65$&$4$&$( 2,1)$ &$79.4266$ \\
		\hline
		$1.05$&$	0.2	$&$13$&$10	$&$(5,1)$&$	459.715$\\
		\hline
		$1.05$&$	0.4	$&$80$&$	10$&$(	4,1)$&$	527.142$\\
		\hline
		$ 1.2$&$0.2$&$15$&$4$&$(2,1	)$&$31.0966$\\
		\hline
		$1.2$&$	0.2$&$	15$&$	20$&$(	9,1)$&$	523.768$\\
		\hline
		$1.2$&$	0.4$&$	175	$&$4$&$(	1,1)$&$	-1.97781$\\
		\hline
		$2$&$	0.2$&$	30$&$	4$&$(	2,1)$&$	5.66053$\\
		\hline
		$2	$&$0.4$&$	60$&$	4$&$(	3,1)$&$	2.0822$\\
		\hline
		$2$&$	0.4	$&$100$&$	4$&$(	2,1)$&$	2.479171$\\
		\hline
		$8$&$	0.4$&$	80$&$	10$&$(	10,5)$&$	6.07011$\\
		
		\hline     
	\end{tabular}
	\caption{Calculation for $p(\lambda_c)$}
	\label{tab:my_label}
\end{table}
Majority of the parameters we have tested show jump transitions. Figure \ref{example eigenvector} shows a plot corresponding to a typical critical eigenvector with $n_c=6,j_c=3$. 
\begin{figure}[htp]
	\centering
	\includegraphics[width=2.5 in]{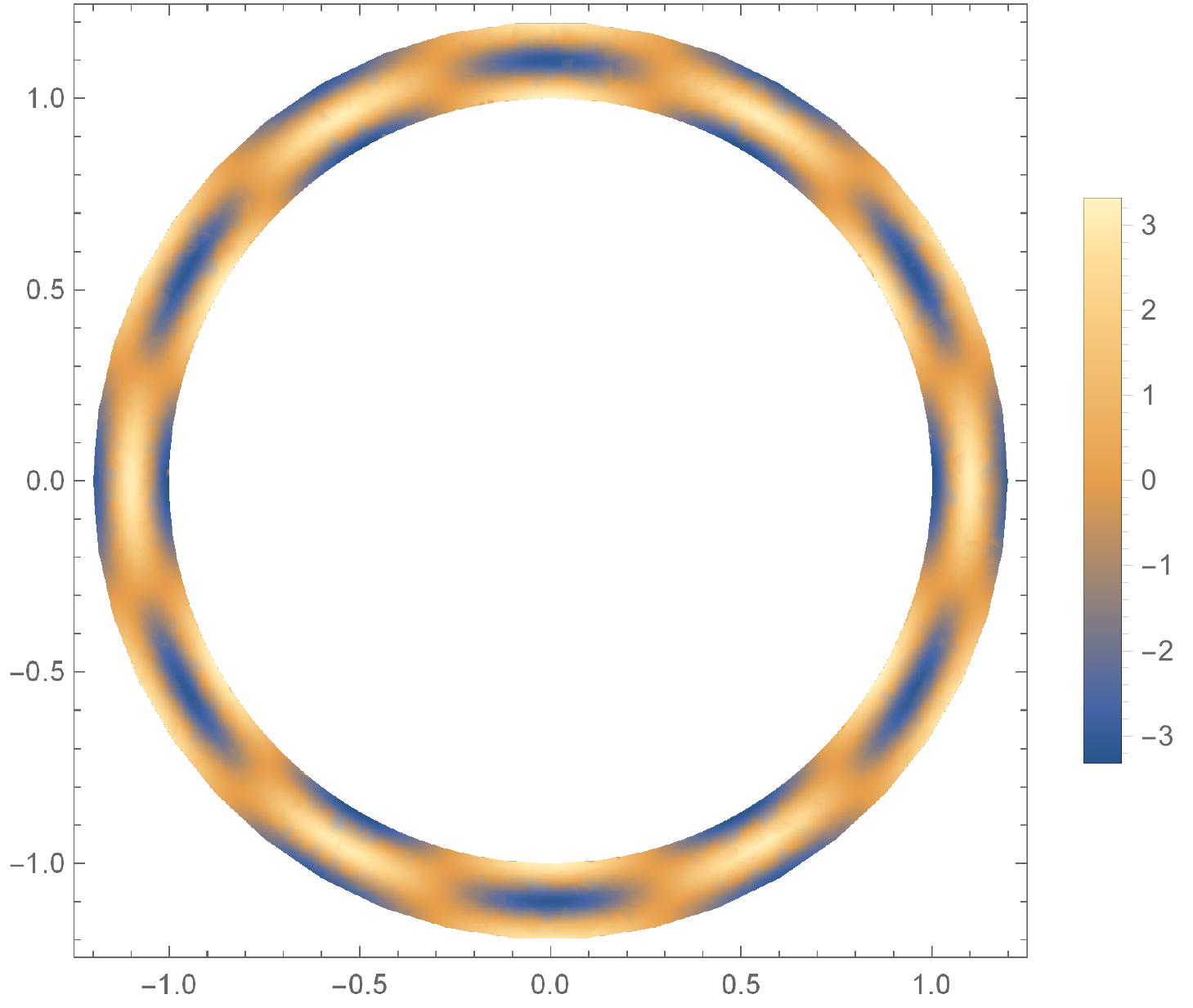}
	\caption{Critical Eigenvector $\cos{n_c\theta}R_{n_c,J_c}(r)$ with $n_c=6, j_c=3$ and $\delta=1.2$, dark or bright regions indicate potential whorl hair growth, in this graph we have around 18 growth regions}
	\label{example eigenvector}
\end{figure}

\section{Conclusions}
We discuss the effect of $\delta,R,\lambda,a,d$ on the system.

From the Remark \ref{remark on multiplicity}, we can see that realistically, when the stem wall of Acetabularia is thin ($\delta$ is small), the crossing modes are two-dimensional and hence Theorem \ref{transition theorem 2d} applies and there is either a jump or continuous type of transition depends on $q(\lambda_c)$. If the stem wall is thick, then a random type of transition can occur and the concentration for substance $u$ or $v$ can display a concentric pattern. 

$R$ affect the general reaction and diffusion rate for both substances, temperature can be a good candidate. From \eqref{intervallambda} and Figure \ref{nc plot}, we can see that greater $R$ (or reaction rate) corresponds to more whorl hair formation because higher wavenumber eigenvectors are involved. This correspond well to the experiment observation from Harrison \cite{harrison1981hair}. 

The effect of $\lambda$ as we are principlely concerned with, is causing a Turing type of transition of the system from a uniform steady state to a spatially variant steady state like those of Figure \ref{example eigenvector}. However, since both continuous and catastrophic type transitions are found from our numerical investigation, the critical eigenvector only faithfully represent the transition state when the type is continuous. For the catastrophic type, we only know the solution will go away from the origin to a global attractor, the precise bifurcated solution is unknown. It is also worth mentioning that $\lambda$ need to be located in a finite interval for the steady state to lose it's stability, as can be seen clearly from Figure \ref{parameter space}.

Parameter $a$ shows the generating rate for substance $u$. Our result shows that when this rate is high, the whole system will be stable and no transition occurs.

The amplitude of $d$ is inversely proportional to the diffusion rate for $u$, by Theorem \ref{parameter thm}, $d$ need to be $>1$ for Turing instability to occur hence $u$ should be a larger molecule than Calcium. Interestingly, our Figure \ref{nc plot} shows that the larger molecule $u$ is, the fewer whorl hair will be generated during transition. 

\section*{Acknowledgments}
The authors are grateful for Prof. Shouhong Wang and anonymous referee for their valuable suggestions. The work of Dongming Yan was supported in part by the China Scholarship Council (CSC).

\providecommand{\href}[2]{#2}
\providecommand{\arxiv}[1]{\href{http://arxiv.org/abs/#1}{arXiv:#1}}
\providecommand{\url}[1]{\texttt{#1}}
\providecommand{\urlprefix}{URL }

\medskip
Received xxxx 20xx; revised xxxx 20xx.
\medskip


\begin{thebibliography}{10}
	
	\bibitem{ashu2013some}
	\newblock A.~M. Ashu,
	\newblock \emph{Some properties of bessel functions with applications to neumann eigenvalues in the unit disc},
	\newblock Bachelor's Theses in Mathematical Sciences, Lund University, 2013.
	
	\bibitem{dumais2000whorl}
	\newblock J.~Dumais and L.~G. Harrison,
	\newblock Whorl morphogenesis in the dasycladalean algae: the pattern formation
	viewpoint,
	\newblock \emph{Philosophical Transactions of the Royal Society of London B:
		Biological Sciences}, \textbf{355} (2000), 281--305.
	
	\bibitem{dumais2000acetabularia}
	\newblock J.~Dumais, K.~Serikawa and D.~F. Mandoli,
	\newblock Acetabularia: a unicellular model for understanding subcellular
	localization and morphogenesis during development,
	\newblock \emph{Journal of plant growth regulation}, \textbf{19} (2000),
	253--264.
	
	\bibitem{goodwin1994leopard}
	\newblock B.~Goodwin,
	\newblock \emph{How the leopard changed its spots: The evolution of complexity},
	\newblock Princeton University Press, 2001.
	
	\bibitem{goodwin1984calcium}
	\newblock B.~Goodwin, J.~Murray and D.~Baldwin,
	\newblock Calcium: the elusive morphogen in acetabularia,
	\newblock in \emph{Proc. 6th Intern. Symp. on Acetabularia. Belgian Nuclear
		Center, CEN-SCK Mol, Belgium}, 
	\newblock (1984), 101--108.
	
	\bibitem{goodwin1985tip}
	\newblock B.~C. Goodwin and L.~Trainor,
	\newblock Tip and whorl morphogenesis in acetabularia by calcium-regulated
	strain fields,
	\newblock \emph{Journal of theoretical biology}, \textbf{117} (1985), 79--106.
	
	\bibitem{harrison1992reaction}
	\newblock L.~G. Harrison,
	\newblock Reaction-diffusion theory and intracellular differentiation,
	\newblock \emph{International journal of plant sciences}, \textbf{153} (1992),
	S76--S85.
	
	\bibitem{harrison1981hair}
	\newblock L.~G. Harrison, J.~Snell, R.~Verdi, D.~Vogt, G.~Zeiss and B.~R.
	Green,
	\newblock Hair morphogenesis inacetabularia mediterranea: temperature-dependent
	spacing and models of morphogen waves,
	\newblock \emph{Protoplasma}, \textbf{106} (1981), 211--221.
	
	\bibitem{henry2006geometric}
	\newblock D.~Henry,
	\newblock \emph{Geometric theory of semilinear parabolic equations}, vol. 840,
	\newblock Springer, 2006.
	
	\bibitem{ma2005bifurcation}
	\newblock T.~Ma and S.~Wang,
	\newblock \emph{Bifurcation theory and applications}, vol.~53,
	\newblock World Scientific, 2005.
	
	\bibitem{ma2014phase}
	\newblock T.~Ma and S.~Wang,
	\newblock \emph{Phase transition dynamics},
	\newblock Springer, 2014.
	
	\bibitem{martynov1975morphogenetic}
	\newblock L.~Martynov,
	\newblock A morphogenetic mechanism involving instability of initial forth,
	\newblock \emph{Journal of theoretical biology}, \textbf{52} (1975), 471--480.
	
	\bibitem{murray2001mathematical}
	\newblock J.~D. Murray,
	\newblock \emph{Mathematical biology II: Spatial models and biomedical applications},
	\newblock 3$^{rd}$ edition, Springer-Verlag, New York, 2001.
	
	\bibitem{siegel2014einige}
	\newblock C.~L. Siegel,
	\newblock {\"U}ber einige anwendungen diophantischer approximationen,
	\newblock in \emph{On Some Applications of Diophantine Approximations}, Springer, (2014), 81--138.
	
	\bibitem{watson1995treatise}
	\newblock G.~N. Watson,
	\newblock \emph{A treatise on the theory of Bessel functions},
	\newblock Cambridge university press, 1995.
	
	\bibitem{you2007global}
	\newblock Y.~You,
	\newblock Global dynamics of the brusselator equations,
	\newblock \emph{Dynamics of Partial Differential Equations}, \textbf{4} (2007),
	167--196.
	
\end{thebibliography}
\end{document}